\documentclass[10pt,a4paper]{article}

\usepackage{epsf,epsfig,amsfonts,amsgen,amsmath,amstext,amsbsy,amsopn,amsthm,amssymb, epstopdf
}
\usepackage{ebezier,eepic}
\usepackage{color}
\usepackage{multirow}
\usepackage{url}
\usepackage{cite}

\newtheorem{theorem}{Theorem}[section]

\newtheorem{lemma}{Lemma}
\newtheorem{false statement}{False statement}

\theoremstyle{definition}

\begin{document}

\title{\bf\Large The bipartite analogue of a classical spanning tree enumeration formula, Boolean functions, and their applications to counting odd spanning trees}
\date{}

\author{Jun Ge\thanks{mathsgejun@163.com}, \ 
Yamin Yu\thanks{3038325763@qq.com} 
\\[2mm]
\small School of Mathematical Sciences, Sichuan Normal University, Chengdu 610066, Sichuan, China
}

\maketitle

\begin{abstract}
Recently, Zheng and Wu defined the concept of odd spanning tree of a graph, meaning a spanning tree in which every vertex has odd degree. 
Similar to Cayley's formula, Feng, Chen and Wu counted the number of odd spanning trees in complete graphs via Prüfer code and the exponential generating function. In this note, we give a simple proof via a classical spanning tree enumeration formula and the Boolean function. 
We also generalize it to complete bipartite graphs.

\medskip


\end{abstract}

\section{Introduction}

Let $G=(V(G), E(G))$ be a graph without self-loops and $\mathcal{T}(G)$ be the collection of spanning trees of $G$.
We usually call $\tau(G)=|\mathcal{T}(G)|$ the number of spanning trees of $G$. Denote by $K_n$, $K_{m,n}$, and $K_{n_1,n_2,\ldots, n_s}$ 
be the complete graph of $n$ vertices, the complete bipartite graph that one partite has $m$ vertices and the other partite has $n$ vertices, 
and the complete $s$-partite graph that each partite has $n_i$ vertices.

Counting spanning trees is a classic problem originating in the 19th century that continues to be a hot topic today. 
Some recent developments can be found in \cite{DG,DKM,GJ,LCY,LY,Yan,YT} for example. 

In 2025, motivated by Gallai's classical theorem (\cite{Lovasz}, Section 5, Problem 17) that the vertices of 
any graph can be partitioned into two sets 
where one induces an even-degree subgraph and the other induces an odd-degree subgraph, and also the concept of 
homeomorphically irreducible spanning trees (HISTs, spanning trees containing no vertices of degree two), 
Zheng and Wu \cite{ZW} introduced the concept of odd spanning trees. An odd spanning tree of a graph $G$ 
is a spanning tree $T$ of $G$ where each vertex has odd degree in $T$.

The famous Cayley's formula \cite{Cayley} states that $\tau(K_n)=n^{n-2}$. 
Let $\tau_o(G)$ be the number of odd spanning trees of $G$. 
Recently, Feng, Chen and Wu counted the number of odd spanning trees in complete graphs via Prüfer code and the exponential generating function as follows. 

\begin{theorem}[\cite{FCW}]\label{Kn}
Let $n$ be a positive number. Then
$$\tau_{o}(K_{n})=\frac{1}{2^{n}}\sum_{k=0}^{n}\binom{n}{k}(2k-n)^{n-2}.$$
\end{theorem}

Note: in \cite{FCW}, the original theorem is of the form
\[
\tau_{o}(K_{n})=
\begin{cases}
0, & \text{if } n \text{ is odd;} \\
\displaystyle\frac{1}{2^{n}}\sum_{k=0}^{n}\binom{n}{k}(2k-n)^{n-2}, & \text{if } n \text{ is even.}
\end{cases}
\]
We slightly revised the theorem since when $n$ is odd, it is easy to check that $\frac{1}{2^{n}}\sum_{k=0}^{n}\binom{n}{k}(2k-n)^{n-2}=0$.

In this note, we we give a simple proof via a classical spanning tree enumeration formula and the Boolean function. 
We also generalize it to complete bipartite graphs.

\section{Preliminaries}

In this section, we give a simple lemma on a specific Boolean function, which will be used in counting odd spanning trees in section \ref{appl}.

Let $y_1,y_2,\dots,y_n$ be Boolean variables(also called sign variables) with $y_i\in\{-1,1\}$ for each $i\in\{1,2,\dots,n\}$, and all variables are independent of one another. Denote the set of all variable combinatuons by $\{\pm1\}^n$ which contains $2^n$ elements in total. Let $a_1,a_2,\dots,a_n$ be real constants and $m\in{\mathbb{N}^\ast}$. Define the function $f: \{\pm1\}^n \to \mathbb{R}$ as the $m$-th power of the linear combination of variables:
	$f(y_1,y_2,\dots,y_n)=(\sum_{i=1}^{n}a_iy_i)^m$, and let $y=(y_1, y_2, \ldots, y_n)$, then the expectation of $f$ is 
$$
	\mathbb{E}[f] = \frac{1}{2^n} \sum_{\substack{y \in \{\pm1\}^n}} \left( \sum_{i=1}^n a_i y_i \right)^m.
$$

In this paper, the expectation operator $\mathbb{E}[\cdot]$ acts on the $m$-th degree homogeneous polynomial $f$ with respect to the random variable $y$. The random variable $y$ satisfies the moment property of origin-symmetric distribution, i.e., for any positive integer $k$, its odd-order moments satisfy $\mathbb{E}[y^k]=0$ and its even-order moments are positive constants. The simplification of expectations relies on the linearity of expectation and the product expectation property of independent random variables.

\begin{lemma}\label{bo}
$$
		\sum_{\substack{y \in \{\pm1\}^n}} \left(\sum_{i=1}^n a_i y_i \right)^m=
			2^n\cdot\sum\limits_{\substack{k_1+k_2+\cdots+k_n=m\\k_i\in 2\mathbb{N} }}\frac{m!}{k_1!k_2!\cdots k_n!}a_1^{k_1}a_2^{k_2}\cdots a_n^{k_n}.
$$	
\end{lemma}

\begin{proof}
	For simplicity, let $\mathcal{Y}=\sum_{\substack{y \in \{\pm1\}^n}} \left( \sum_{i=1}^n a_i y_i \right)^m$. Then $\mathcal{Y}=2^n\mathbb{E}[(\sum_{i=i}^{n}a_iy_i)^m]$. From the expansion of multivariate polynomial, it can be derived that $$\left(\sum_{i=1}^{n}a_iy_i\right)^m=\sum_{\substack{k_1+k_2+\cdots+k_n=m\\k_i\geq0 }}\frac{m!}{k_1!k_2!\cdots k_n!}(a_1y_1)^{k_1}(a_2y_2)^{k_2}\cdots(a_ny_n)^{k_n},$$
	Based on the fundamental properties of expectation, we obtain that
	\begin{align*}
		\mathcal{Y} &=2^n\cdot\mathbb{E}[\sum_{\substack{k_1+k_2+\cdots+k_n=m\\k_i\geq0 }}\frac{m!}{k_1!k_2!\cdots k_n!}(a_1y_1)^{k_1}(a_2y_2)^{k_2}\cdots(a_ny_n)^{k_n}]\\
		&=2^n\cdot\sum_{\substack{k_1+k_2+\cdots+k_n=m\\k_i\geq0 }}\frac{m!}{k_1!k_2!\cdots k_n!}\mathbb{E}[(a_1y_1)^{k_1}]\mathbb{E}[(a_2y_2)^{k_2}]\cdots\mathbb{E}[(a_ny_n)^{k_n}]\\
		&=2^n\cdot\sum_{\substack{k_1+k_2+\cdots+k_n=m\\k_i\geq0 }}\frac{m!}{k_1!k_2!\cdots k_n!}a_1a_2\cdots a_n\mathbb{E}[y_1^{k_1}]\mathbb{E}[y_2^{k_2}]\cdots\mathbb{E}[y_n^{k_n}],
	\end{align*}	
	Since $y_i$ is a Boolean variable, $\mathbb{E}[y_i^{k_i}]=1$ if $k_i$ is even and $\mathbb{E}[y_i^{k_i}]=0$ if $k_i$ is odd. Thus, 
$$
		\sum_{\substack{y \in \{\pm1\}^n}} \left(\sum_{i=1}^n a_i y_i \right)^m=
			2^n\cdot\sum\limits_{\substack{k_1+k_2+\cdots+k_n=m\\k_i\in 2\mathbb{N} }}\frac{m!}{k_1!k_2!\cdots k_n!}a_1^{k_1}a_2^{k_2}\cdots a_n^{k_n}.
$$	
\end{proof}

\section{Bipartite analogue of a classical spanning tree
enumeration formula}

The following theorem is well-known and can be proved by induction on the number of vertices. 

\begin{theorem}[\cite{Berge,Lovasz}]\label{Kn-degree}
Let $d_1,d_2,\dots,d_n$ be positive integers summing up to $2n-2$. Then the number of spanning trees of $K_n$ in which the vertex $i$ has degree exactly $d_i$ for all $i=1,2,\dots,n$ equals$$
	\dfrac{(n-2)!}{(d_1-1)!(d_2-1)!\cdots(d_n-1)!}.$$
\end{theorem}

We now prove the bipartite analogue of this lemma. This result will be applied in Section 4 to count odd spanning trees in complete bipartite graphs, and is also of independent interest in its own right.

\begin{theorem}\label{Kmn-degree}
Let $K_{m,n}$ be the complete bipartite graph with bipartition $A=\{u_1,u_2,\ldots,u_m\}$ and $B=\{v_1,v_2,\ldots,v_n\}$. Let $a_i$ and $b_j$ be positive integers which $\sum_{i=1}^{m}a_i=\sum_{j=1}^{n}b_j=m+n-1$. Then the number of spanning trees of $K_{m,n}$ in which the vertex $u_i$ and $v_j$ have degree exactly $a_i$ and $b_j$ for $i\in\{1,2,\ldots,m\}$ and $j\in\{1,2,\ldots,n\}$ equals 
    $$
	\frac{(m-1)!(n-1)!}{\prod_{i=1}^{m}(a_i-1)!\prod_{j=1}^{n}(b_j-1)!},
	$$
\end{theorem}
\begin{proof}
	We use induction on $m+n$. For $m=1$ or $n=1$, $K_{m,n}$ is a star and the statement is trivial. 
Since $\sum_{i=1}^{m}a_i=\sum_{j=1}^{n}b_j=m+n-1$,  $\sum_{i=1}^{m}a_i+\sum_{j=1}^{n}b_j=2(m+n)-2<2(m+n)$, 
there must exist a leaf, and we may assume that $a_m=1$ where $m>1$. 
Remove $u_m$, in any tree under consideration, $u_m$ is adjacent to some $v_j$, $1\leq j\leq n$ and the removal of $u_m$ results in another tree on $\{u_1,\dots,u_{m-1},v_1,\dots,v_n\}$ with degrees $a_1,\dots,a_{m-1},b_1,\cdots,b_j-1,\dots,b_n$. 
By the induction hypothesis, the number of trees in $K_{m-1,n}$ that $\{u_1,\dots,u_{m-1},v_1,\dots,v_n\}$ has degrees $a_1,\dots,a_{m-1},b_1,\cdots,b_j-1,\dots,b_n$ is
	$$\frac{(m-2)!(n-1)!}{(a_1-1)!\cdots(a_{m-1}-1)!(b_1-1)!\cdots(b_j-2)!\cdots(b_n-1)!}
		=\frac{(m-2)!(n-1)!(b_j-1)}{\prod_{i=1}^{m}(a_i-1)!\prod_{j=1}^{n}(b_j-1)!}.
	$$	
	Thus, the number of trees on $\{u_1,\dots,u_{m},v_1,\dots,v_n\}$ with degrees $a_1,\dots,a_{m},b_1,\dots,b_n$ is

\begin{eqnarray*}
\  &  & \sum_{j=1}^{n}\frac{(m-2)!(n-1)!(b_j-1)}{\prod_{i=1}^{m}(a_i-1)!\prod_{j=1}^{n}(b_j-1)!} \\
& = & \frac{(m-2)!(n-1)!}{\prod_{i=1}^{m}(a_i-1)!\prod_{j=1}^{n}(b_j-1)!}\sum_{j=1}^{n}(b_j-1) \\
& = & \frac{(m-1)!(n-1)!}{\prod_{i=1}^{m}(a_i-1)!\prod_{j=1}^{n}(b_j-1)!}.
\end{eqnarray*}
\end{proof}

Now we give a new proof on the number of complete bipartite graphs.

\begin{theorem}\cite{FS}
$$\tau(K_{m,n})=m^{n-1}n^{m-1}.$$
\end{theorem}

\begin{proof}
The notation and symbols used below are consistent with those in Theorem \ref{Kmn-degree}. 
Let $a_i'=a_i-1$ and $b_j'=b_j-1$ for $i\in\{1,2,\dots,m\}$ and $j\in\{1,2,\dots,n\}$. 
Then $\sum_{i=1}^{m}a_i'=n-1$, $\sum_{j=1}^{n}b_j'=m-1$. The number of spanning trees of $K_{m,n}$ is
  
\begin{eqnarray*}
\ \tau(K_{m,n}) & = & \sum_{\substack{a_1'+\cdots+a_m'=n-1\\b_1'+\cdots+b_n'=m-1}}\frac{(m-1)!(n-1)!}{\prod_{i=1}^{m}(a_i-1)!\prod_{j=1}^{n}(b_j-1)!} \\
& = & \sum_{\substack{a_1'+\cdots+a_m'=n-1}}\frac{(n-1)!}{\prod_{i=1}^{m}(a_i-1)!}
      \cdot\sum_{\substack{b_1'+\cdots+b_n'=m-1}}\frac{(m-1)!}{\prod_{j=1}^{m}(b_i-1)!} \\
& = & (\underbrace{1+1+\cdots+1}_{m})^{n-1}(\underbrace{1+1+\cdots+1}_{n})^{m-1} \\
& = & m^{n-1}n^{m-1}.
\end{eqnarray*}

  \end{proof}

\section{Applications to counting odd spanning trees}\label{appl}

We first give a simple proof of Theorem \ref{Kn} via Theorem \ref{Kn-degree} and the Boolean function.

\noindent
{\bf Proof of Theorem \ref{Kn}.}
By Theorem \ref{Kn-degree}, the number of odd spanning trees of $K_n$ in which 
the vertex $i$ has degree exactly $d_i=2d_i'+1$ for all $i=1,2,\dots,n$ equals$$
	\dfrac{(n-2)!}{(2d_1')!(2d_2')!\cdots(2d_n')!},$$
	where $d_i\geq0$, and $\sum_{i=1}^{n}(2d_i')=n-2$.
Therefore, 
$$\tau_o(K_n)=\sum_{\substack{k_1+k_2+\cdots+k_n=n-2\\k_i\in 2\mathbb{N} }}\frac{(n-2)!}{k_1!k_2!\cdots k_n!}.$$
It follows from Lemma \ref{bo} that
$$
	\sum_{\substack{\varepsilon_i \in \{\pm1\}}}(\varepsilon_1x_1+\varepsilon_2x_2+\cdots\varepsilon_nx_n)^{n-2}=	2^n\cdot\sum_{\substack{k_1+k_2+\cdots+k_n=n-2\\k_i\in 2\mathbb{N} }}\frac{(n-2)!}{k_1!k_2!\cdots k_n!}x_1^{k_1}x_2^{k_2}\cdots x_n^{k_n}.
	$$
Let $x_1=x_2=\cdots=x_n=1$, we obtain that
       $$
		\tau_o(K_n)=\frac{1}{2^n}	\sum_{\substack{\varepsilon_i \in \{\pm1\}}}(\varepsilon_1+\varepsilon_2+\cdots\varepsilon_n)^{n-2}\\
		=\frac{1}{2^n}\sum_{k=0}^{n}\binom{n}{k}(2k-n)^{n-2}.
	  $$
{\hfill$\Box$}

Now we consider the number of odd spanning trees in complete bipartite graphs. 

\begin{theorem}\label{Kmn}
Let $m$ and $n$ be positive numbers. Then
$$\tau_o(K_{m,n})=\frac{1}{2^{m+n}}\left[\sum_{i=0}^{m}\binom{m}{i}(2i-m)^{n-1}\right]\left[\sum_{j=0}^{n}\binom{n}{j}(2j-n)^{m-1}\right].$$	
\end{theorem}

\begin{proof}
	The notation and symbols used below are consistent with those in Theorem \ref{Kmn-degree}. 
By Theorem \ref{Kmn-degree}, the number of odd spanning trees of $K_{m,n}$ 
in which the vertex $u_i\in A$ and $v_j\in B$ have degree exactly $d(u_i)=2d'(u_i)+1$ and $d(v_j)=2d'(v_j)+1$ 
for all $i\in\{1,2,\dots,m\}$ and $j\in\{1,2,\dots,n\}$ equals
$$
	\frac{(m-1)!(n-1)!}{\prod_{i=1}^{m}2d'(u_i)!\prod_{j=1}^{n}2d'(v_j)!},
	$$
	where all $d'\geq0$, $\sum_{i=1}^{m}2d'(u_i)=n-1$ and $\sum_{j=1}^{n}2d'(v_j)=m-1$. 
Therefore, 
$$\tau_o(K_{m,n})=\sum_{\substack{k_1+k_2+\cdots+k_m=n-1\\k_i\in 2\mathbb{N} }}
                  \sum_{\substack{l_1+l_2+\cdots+l_n=m-1\\l_i\in 2\mathbb{N} }}\frac{(m-1)!(n-1)!}{\prod_{i=1}^{m}k_i!\prod_{j=1}^{n}l_j!}.$$

It follows from Lemma \ref{bo} that
\begin{eqnarray*}
\ &  & \sum_{\substack{\alpha_i \in \{\pm1\}}}(\alpha_1x_1+\cdots+\alpha_mx_m)^{n-1}
        \cdot\sum_{\substack{\beta_i \in \{\pm1\}}}(\beta_1y_1+\cdot+\beta_ny_n)^{m-1} \\
& = & 2^m\sum_{\substack{k_1+\cdots+k_m=n-1\\k_i\in 2\mathbb{N}}}\frac{(n-1)!}{k_1!\cdots k_m!}x_1^{k_1}\cdots x_m^{k_m}
      \cdot2^n\sum_{\substack{l_1+\cdots+l_n=m-1\\l_j\in 2\mathbb{N}}}\frac{(m-1)!}{l_1!\cdots l_m!}y_1^{l_1}\cdots y_m^{l_m} \\
& = & 2^{m+n}\left[\sum_{\substack{k_1+\cdots+k_m=n-1\\k_i\in 2\mathbb{N}}}\frac{(n-1)!}{k_1!\cdots k_m!}x_1^{k_1}\cdots x_m^{k_m}\right]
      \left[\sum_{\substack{l_1+\cdots+l_n=m-1\\l_j\in 2\mathbb{N}}}\frac{(m-1)!}{l_1!\cdots l_m!}y_1^{l_1}\cdots y_m^{l_m}\right]
\end{eqnarray*}

Set $x_1=\cdots=x_m=y_1=\cdots=y_n=1$, we can obtain that
	\begin{align*}
		\tau_o(K_{m,n})&=\dfrac{1}{2^{m+n}}\sum_{\substack{\alpha_i \in \{\pm1\}}}(\alpha_1+\cdots+\alpha_m)^{n-1}\cdot\sum_{\substack{\beta_i \in \{\pm1\}}}(\beta_1+\cdot+\beta_n)^{m-1}\\
		&=\dfrac{1}{2^{m+n}}\sum_{i=1}^{m}\binom{m}{i}(2i-m)^{n-1}\sum_{j=1}^{n}\binom{n}{j}(2j-n)^{m-1}.
	\end{align*}
	This complete the proof.
\end{proof}

Note that if at least one of $m$ and $n$ is even, then $\tau_o(K_{m,n})=0$. This can be easily deduced by the handshaking lemma.

\section*{Acknowledgements}
This research is supported by the National Natural Science Foundation of China (No. 12371355).

\section*{Declarations}
{\bf Conflicts of interest} ~There are no conflicts of interests or competing interests.


\begin{thebibliography}{10}
\bibitem{Berge}
C. Berge,
Graphs and hypergraphs,
North-Holland Math. Library, Vol. 6,
North-Holland Publishing Co., Amsterdam-London; American Elsevier Publishing Co., Inc., New York, 1973.

\bibitem{Cayley}
A. Cayley,
A theorem on trees,
\emph{Quart. J. Pure Appl. Math.} {\bf 23} (1889), 376--378.

\bibitem{DG}
F.M. Dong, J. Ge,
Counting spanning trees in a complete bipartite graph which contain a given spanning forest,
\emph{J. Graph Theory} {\bf 101}(1) (2022), 79--94.

\bibitem{DKM}
A.M. Duval, C.J. Klivans, J.L. Martin,
Simplicial matrix-tree theorem.
\emph{Trans. Amer. Math. Soc.} {\bf 361} (2009),
6073--6114.

\bibitem{FCW}
Y.D. Feng, Y.W. Chen, B. Wu, 
The number of odd spanning trees in the complete graphs, 
\emph{Discrete Appl. Math.} {\bf 383} (2026), 295--299.

\bibitem{FS} 
M. Fiedler, J. Sedl\'a\v{c}ek, 
\"{U}ber Wurzelbasen von gerichteten Graphen, 
\emph{\v{C}asopis P\v{e}st. Mat.} {\bf 83} (1958), 214--225.

\bibitem{GJ}
H.L. Gong, X.A. Jin,
A simple formula for the number of spanning trees of line graphs,
\emph{J. Graph Theory} {\bf 88} (2018), 294--301.

\bibitem{LCY}
D.Y. Li, W.X. Chen, W.G. Yan, 
Enumeration of spanning trees of complete multipartite graphs containing a fixed spanning forest,
\emph{J. Graph Theory} {\bf 104} (2023), 160--170.

\bibitem{LY}
D.Y. Li, W.G. Yan, 
A variant of the Teufl‐Wagner formula and applications,
\emph{J. Graph Theory} {\bf 109} (2025), 68--75.

\bibitem{Lovasz}
L. Lov\'{a}sz,
Combinatorial problems and exercises,
North-Holland Publishing Co., Amsterdam, 1979.

\bibitem{Yan}
W.G. Yan,
On the number of spanning trees of some irregular line graphs,
\emph{J. Combin. Theory Ser. A} {\bf 120} (2013), 1642--1648.

\bibitem{YT}
C. Yang, T. Tian,  
Counting spanning trees of multiple complete split-like graph containing a given spanning forest, 
\emph{Discrete Math.} {\bf 348(2)} (2025) 114300.


\bibitem{ZW}
J. Zheng, B. Wu,  
Odd spanning trees of a graph, (2025) arXiv 2503.17676.









\end{thebibliography}
\end{document}